\documentclass[10pt, a4paper
]{amsart}

\usepackage{amsfonts, amsthm, amsmath, amscd, amssymb}
\usepackage[draft=false]{hyperref}
\usepackage[T1]{fontenc}
\usepackage{ae}
\usepackage[all]{xy}

\renewcommand{\le}{\leqslant}
\renewcommand{\ge}{\geqslant}

\newtheorem{theorem}{Theorem}
\newtheorem{lemma}[theorem]{Lemma}
\newtheorem{proposition}[theorem]{Proposition}

\theoremstyle{definition}

\numberwithin{equation}{section}

\renewcommand\rho{\varrho}

\title{A Bombieri-Vinogradov-type theorem with prime power moduli}


\author{Stephan~Baier}
\address{Stephan~Baier\\
	Ramakrishna Mission Vivekananda Educational Research Institute\\
	Department of Mathematics\\
	G.\ T.\ Road, PO~Belur Math, Howrah, West Bengal~711202\\
	India}
\email{stephanbaier2017@gmail.com}
\urladdr{https://www.researchgate.net/profile/Stephan\_Baier2}

\author{Sudhir Pujahari}
\address{School of Mathematical Sciences, National Institute of Science Education and Re- search, Bhubaneswar, An OCC of Homi Bhabha National Institute,  P. O. Jatni,  Khurda 752050, Odisha, India.}
\email{spujahari@niser.ac.in}
\urladdr{https://sites.google.com/site/sudhirkumarpujahari/home}

\begin{document}
\maketitle

\begin{abstract} In 2020, Roger Baker \cite{Bak} proved a result on the exceptional set of moduli in the prime number theorem for arithmetic progressions of the following kind. Let $\mathcal{S}$ be a set of pairwise coprime moduli $q\le x^{9/40}$. Then the primes $l\le x$ distribute as expected in arithmetic progressions mod $q$, except for a subset of $\mathcal{S}$ whose cardinality is bounded by a power of $\log x$. We use a $p$-adic variant Harman's sieve to extend Baker's range to 
$q\le x^{1/4-\varepsilon}$ if $\mathcal{S}$ is restricted to prime powers $p^N$, where $p\le (\log x)^C$ for some fixed but arbitrary $C>0$. For large enough $C$, we thus get an almost all result. Previously, an asymptotic estimate for $\pi(x;p^N,a)$ of the expected kind, with $p$ being an odd prime, was established in the wider range $p^N\le x^{3/8-\varepsilon}$ by Barban, Linnik and Chudakov \cite{BLC}. Gallagher \cite{Gal} extended this range to $p^N\le x^{2/5-\varepsilon}$ and Huxley \cite{Hux2} improved Gallagher's exponent to $5/12$. A lower bound of the correct order of magnitude was recently established by Banks and Shparlinski \cite{BaS} for the even wider range $p^N\le x^{0.4736}$. However, all these results hold for {\it fixed} primes $p$, and the $O$-constants in the relevant estimates depend on $p$. Therefore, they do not contain our result. In a part of our article, we describe how our method relates to these results.   
\end{abstract}

\bigskip

\tableofcontents

\section{Main results} \label{results}
Let $\Lambda(n)$ be the von Mangoldt function, $q$ be a positive integer and $a$ be an integer coprime to $q$. For $x\ge 2$ let 
$$
E(x;q,a):=\sum\limits_{\substack{n\le x\\ n \equiv a \bmod{q}}} \Lambda(n)- \frac{x}{\varphi(q)}
$$
be the error term in the prime number theorem for the arithmetic progression $a \bmod{q}$, and set
$$
E(x,q):=\max\limits_{\substack{a\\ (a,q)=1}} |E(x;q,a)| \quad \mbox{and} \quad E^{\ast}(x,q):=\max\limits_{y\le x} |E(y,q)|.
$$
Set $\mathcal{L}:=\log x$ throughout the sequel. The Bombieri-Vinogradov theorem implies that 
$$
E^{\ast}(x,q)\le \frac{x}{\varphi(q)\mathcal{L}^A}
$$
for all integers $q\in (Q,2Q]$ with at most $O(Q\mathcal{L}^{-A})$ exceptions, provided that $Q\le x^{1/2}\mathcal{L}^{-2A-6}$. Under GRH, the above inequality would hold for all $q\le x^{1/2-\varepsilon}$. Roger Baker \cite{Bak} proved the following result, significantly restricting the set of exceptional moduli. The cost is that $Q$ is restricted to a smaller interval as well. 

\begin{theorem}[Baker] Let $Q\le x^{9/40}$. Let $\mathcal{S}$ be a set of pairwise relatively prime integers in $(Q,2Q]$. Then the number of $q$ in $\mathcal{S}$ for which
$$
E^{\ast}(x,q)>\frac{x}{\varphi(q)\mathcal{L}^A}
$$ 
is $O(\mathcal{L}^{34+A})$. 
\end{theorem}

We extend the above range to $Q\le x^{1/4-\varepsilon}$ for the case when $\mathcal{S}$ consists of powers of primes $p$ such that $p\le \mathcal{L}^C$ for some fixed but arbitrary constant $C>0$. Our bound for the cardinality of the exceptional set will be slightly different: We get $16+2A$ in place of $34+A$ in the exponent of the logarithm (see Theorem \ref{mainresult} below). Here we consider the prime counting function instead of the summatory function of the von Mangoldt function. 
Let 
$$
\pi(x;q,a):=\sharp\{p\le x : p \mbox{ prime and } p\equiv a \bmod{q}\},
$$
$$
F(x;q,a):=\pi(x;q,a)- \frac{1}{\varphi(q)} \int\limits_2^x \frac{dt}{\log t}
$$
and 
$$
F(x,q):=\max\limits_{\substack{a\\ (a,q)=1}} |F(x;q,a)| \quad \mbox{and} \quad  F^{\ast}(x,q):=\max\limits_{y\le x} |F(y,q)|.
$$ 
We prove the following.

\begin{theorem} \label{mainresult} Fix $C\ge 6$ and $\varepsilon>0$. Assume that $x^{\varepsilon}\le Q\le x^{1/4-\varepsilon}$. Let $\mathcal{S}\subseteq (Q,2Q]\cap \mathbb{N}$ be a set of powers of distinct primes $p\le \mathcal{L}^C$. Then the number of $q$ in $\mathcal{S}$ for which 
$$
F^{\ast}(x,q)>\frac{x}{\varphi(q)\mathcal{L}^A}
$$
is $O_{C,\varepsilon}(\mathcal{L}^{16+2A})$. 
\end{theorem}

We need to make sure that it is actually {\it possible} that $\mathcal{S}$ has a cardinality much larger than $\mathcal{L}^{16+2A}$. This is the case if $C$ is sufficiently large as compared to $A$ and the relevant set $\mathcal{P}$ of primes $p\le \mathcal{L}^C$ is not too sparse. We have the following lower bound for the number of prime powers in question.

\begin{lemma} \label{SQ} Fix $C\ge 6$ and $\varepsilon>0$. Assume that $x^{\varepsilon}\le Q\le x$. Then if $x$ is large enough, the number $S(Q)$ of prime powers $p^N$ in the interval $(Q,2Q]$ such that $p\le \mathcal{L}^C$ satisfies
$$
S(Q)\gg \mathcal{L}^{C-1}.
$$ 
\end{lemma}

\begin{proof} Let 
$$
n:=\left\lceil \frac{\log 2Q}{\log \mathcal{L}^{C}} \right\rceil.
$$
Then $(2Q)^{1/n}\le \mathcal{L}^C$ and hence
$$
S(Q)\ge \pi\left((2Q)^{1/n}\right)-\pi\left(Q^{1/n}\right)
$$
since every power $p^n$ of a prime $p\in (Q^{1/n},(2Q)^{1/n}]$ is  an element of $(Q,2Q]$.  In the above inequality $\pi(x)$ is the number of primes less than equal to $x$.
We note that
$$
(2Q)^{1/n}-Q^{1/n}=\left(2^{1/n}-1\right)Q^{1/n}=\left(\frac{\log 2}{n} + O\left(\frac{1}{n^2}\right)\right)Q^{1/n}.
$$
Now we invoke Huxley's prime number theorem \cite{Hux} which ensures that 
\begin{equation} \label{Huxl}
\pi(x+y)-\pi(x) \sim \frac{y}{\log x} \quad \mbox{ as } x\rightarrow\infty,
\end{equation}
provided that $y\ge x^{7/12+\varepsilon}$. Hence, we have 
\begin{equation} \label{Huxley}
\pi\left((2Q)^{1/n}\right)-\pi\left(Q^{1/n}\right)\sim \frac{\log 2}{n}\cdot \frac{Q^{1/n}}{\log Q^{1/n}} = \frac{\log 2}{\log Q}\cdot Q^{1/n},
\end{equation}
provided that 
$$
\frac{Q^{1/n}}{n} \ge Q^{2/(3n)}
$$
and $Q$ is large enough. The above inequality is equivalent to 
\begin{equation} \label{easy}
n\le Q^{1/(3n)}.
\end{equation}
Moreover,
$$
n\le \frac{2\log Q}{\log \mathcal{L}^C}
$$
if $Q\ge x^{\varepsilon}$ and $x$ is large enough, and thus 
$$
Q^{1/(3n)} \ge Q^{(\log \mathcal{L}^C)/(6\log Q)} = e^{(\log \mathcal{L}^C)/6}=\mathcal{L}^{C/6}.
$$
Therefore, \eqref{easy} holds if 
$$
2\log Q \le \mathcal{L}^{C/6}\log \mathcal{L}^C. 
$$
This is the case under the conditions $C\ge 6$ and $Q\le x$ in the lemma if $x$ is large enough. Hence, \eqref{Huxley} holds. We further observe that
$$
Q^{1/n}\ge Q^{1/\left(1+\log(2Q)/\log \mathcal{L}^C\right)} = Q^{\left(1+O\left(\log \mathcal{L}^C/\log Q\right)\right) \log \mathcal{L}^C/\log Q}
\gg \mathcal{L}^{C}  
$$  
under the condition $Q\ge x^{\varepsilon}$ of the lemma if $x$ is large enough. Hence, we obtain
$$
\mathcal{S}(Q)\gg \frac{\mathcal{L}^C}{\log Q}\gg \mathcal{L}^{C-1},
$$
which completes the proof.
\end{proof}

It follows that if we take $\mathcal{P}$ as large as possible, i.e. equal to the set of all primes $p\le \mathcal{L}^C$, then Theorem \ref{mainresult} is non-trivial provided that $C>17+2A$. 

For any {\it fixed} odd prime number $p$, Barban, Linnik and Chudakov \cite{BLC} proved that an asymptotic of the form
\begin{equation} \label{BaLiCh}
\pi(x;p^N,a)= \frac{1}{\varphi(p^N)} \int\limits_2^x \frac{dt}{\log t} \cdot \left(1+O\left(\frac{1}{\mathcal{L}^A}\right)\right)
\end{equation}
holds for {\it single} residue classes modulo a prime power $p^N$ in the wider range $p^N\le x^{3/8-\varepsilon}$. The exponent $3/8$ was improved to $2/5$ by Gallagher \cite{Gal} and to $5/12$ by Huxley \cite{Hux2}. In recent work, Banks and Shparlinski \cite{BaS} obtained a further improvement to $0.4736$, with a lower bound instead of the asymptotic estimate \eqref{BaLiCh}. However, the $O$-constants in all these results depend on $p$ and not just on $A$. Therefore, they do not contain our Theorem \ref{mainresult}. 

We also mention that Iwaniec \cite{Iwa} extended Gallagher's work to power-full moduli and that Gallagher's result is a special case of a later result by Guo \cite{Guo} who elaborated on work by Elliott \cite{Ell} in which a Bombieri-Vinogradov-type theorem for moduli divisible by a power-full number was established. Gallagher, in turn, made use of Postnikov's \cite{Pos} important investigations of characters to prime power moduli which led to many new results. \\

{\bf Acknowledgements:} The authors would like to thank Igor Shparlinski for useful discussions.   The authors would also like to thank Akshaa Vatwani for her comment on an earlier version of the paper and the anonymous referee for careful reading and many useful comments.  The second author was supported by ERCIM `Alain Bensoussan' Fellowship Programme while some part of the project was carried out.  

\section{Remarks on our method}
In this section, we indicate the idea of our approach and  describe how it relates to the above-mentionded works \cite{BLC}, \cite{Gal}, \cite{Hux2} and \cite{BaS}. Our method to prove Theorem \ref{mainresult} relies on Harman's asymptotic sieve (see Proposition \ref{Harmanasymp} below) and the large sieve. The idea is to compare the number of primes in the sets 
$$
\mathcal{A}:=\{n\le y \ :\ n\equiv e \bmod{p^N}\}
$$ 
and 
$$
\mathcal{B}:=\{n\le y \ :\  n\equiv d \bmod{p}\}
$$
with $2\le y\le x$, where $e\equiv d \bmod{p}$ so that $\mathcal{A} \subseteq \mathcal{B}$. If $p\le \mathcal{L}^C$, then the Siegel-Walfisz theorem tells us that the number of primes in $\mathcal{B}$ is as expected. The idea to use a version of Harman's sieve to sift for primes in residue classes is not new. Harman himself did this in his papers \cite{HarCarmI} and \cite{HarCarmII} on Carmichael numbers for more general moduli $q$.  Here we focus on prime power moduli. Our method may therefore be viewed as a $p$-adic variant of Harman's sieve. We sieve for primes in a set $\mathcal{A}$ of integers whose $p$-adic distance to a given integer $e$ is small. Their cardinality is compared to that of the primes in a set $\mathcal{B}$ of integers whose $p$-adic distance to $e$ is potentially large.  

If we tried to apply Theorem \ref{Harmanasymp} below to a single modulus $p^N$, we would not obtain the type II information required to establish that
$$
\left|F\left(y;p^N,e\right)\right| \ll \frac{x}{\varphi(p^N)\mathcal{L}^A}
$$ 
for any fixed $A>0$.
Only if we consider a set of sufficiently many residue classes $e \bmod{p^N}$, we obtain the type II information required to establish an almost all result as in Theorem \ref{mainresult} by using the large sieve. One runs into the same issue if one tries to apply Theorem \ref{Harmanasymp} directly to detect primes in short intervals. This suggests that a $p$-adic version of the sieves developed by Harman and his co-authors to detect primes in short intervals would give analog results for primes in arithmetic progressions to prime power moduli. Below we have a closer look at this relation. 

In \cite[sections 7.2,7.3]{HarPri}, Harman describes how to overcome the said issue of lacking type II information by a refinement of his sieve, allowing him to recover Huxley's prime number theorem for short intervals (see \eqref{Huxl} above). Here one takes $\mathcal{A}=I\cap \mathbb{N}$ and $\mathcal{B}=I_1\cap \mathbb{N}$, where $I$ is a short interval and $I_1$ is a large interval. He decomposes the functions measuring the cardinalities of the sifted sets into trilinear sums of the form
$$
\sum\limits_{\substack{m\sim M\\ n\sim N\\ r\sim R\\ mnr \in \mathcal{A}}} a_mb_nc_r,
$$
where the coefficients $c_r$ satisfy a condition of the form
$$
\sum\limits_{r\sim R} c_r r^{it} \ll \frac{R}{\mathcal{L}^A}
$$
for all $t$ in a suitable range. This condition holds if $c_r=c(r)$ is the indicator function of the primes, in particular. It allows to approximate the above trilinear sums efficiently when the condition $mnr\in \mathcal{A}$ is picked out using Perron's formula, leading to an extra factor of $(mnr)^{it}=m^{it}n^{it}r^{it}$. We may try a similar approach to our problem. In our case, $mnr\in \mathcal{A}$ is a congruence condition which is, most naturally, picked out using Dirichlet characters modulo $p^N$, leading to an extra factor of $\chi(mnr)=\chi(m)\chi(n)\chi(r)$. The above condition on the coefficients $c_r$ then becomes
$$
\sum\limits_{r\sim R} c_r\chi(r) \ll \frac{R}{\mathcal{L}^A}, 
$$ 
which takes the from
\begin{equation*}
\sum\limits_{\substack{r\sim R \\ r \mbox{\ \rm \scriptsize prime}}} \chi(r)\ll \frac{R}{\mathcal{L}^A}
\end{equation*}
for $c_r=c(r)$ the indicator function of the primes. In fact, one needs, more generally, bounds for hybrid sums over primes of the form  
 \begin{equation} \label{form}
\sum\limits_{\substack{r\sim R \\ r \mbox{\ \rm \scriptsize prime}}} \chi(r)r^{it}\ll \frac{R}{\mathcal{L}^A},
\end{equation}
where the extra term $r^{it}$ comes in by a lemma which Harman terms ``cosmetic surgery'' (see Lemma \ref{CoSu} in the appendix), designed to remove dependencies of summation variables $a$ and $b$ by  inequalities of the form $a<b$. The mean value estimates for Dirichlet polynomials which are of key importance in Harman's method in \cite[sections 7.2,7.3]{HarPri} need to be replaced by corresponding large sieve estimates for character sums
$$
\sum\limits_{n\sim N} a_n\chi(n)
$$  
with general coefficients $a_n$ (in particular, large value estimates for character sums). 

In general, a bound like in \eqref{form} is not available if the conductor $q$ of $\chi$ is large. However, in the case of characters for large powers of a {\it fixed} odd prime $p$, such bounds are known to hold due to the aforementioned works \cite{BLC}, \cite{Gal}, \cite{Hux2} and \cite{BaS} in which particularly wide zero free regions for the corresponding Dirichlet $L$-functions were utilized.  If we want to establish the expected {\it asymptotic} for $\pi(x;p^N,a)$, we can therefore achieve this by a $p$-adic analogue of Harman's method in \cite[sections 7.2,7.3]{HarPri} as outlined above. This should give the said asymptotic if $p^N\le x^{5/12-\varepsilon}$, which recovers Huxley's result in \cite{Hux2}. We note that the works \cite{BLC}, \cite{Gal} and \cite{Hux2} do not use sieve methods but only depend on information about zeros of $L$-functions (zero-free regions, zero density estimates). If we allow for a lower bound instead of an asymptotic estimate, the exponent $5/12-\varepsilon$ is superseded by the exponent $0.4736$ achieved by Banks and Shparlinski \cite{BaS}. They obtained their result by a combination of a result in \cite{HarCarmII} and improved zero-free regions derived from new estimates for short character sums. Since their work uses \cite[Theorem 1.2]{HarCarmII}, their method depends on a version of Harman's sieve for arithmetic progressions, as in our setting. If a complete $p$-adic analogue of the sieve method used in \cite{BHP} to detect primes in short intervals (also described in \cite[section 7.4.]{HarPri}) were available, the exponent $0.4736$ achieved by Banks and Shparlinski could be further improved to $19/40=0.475$. A crucial role in the method of Baker, Harman and Pintz plays Watt's mean value theorem for Dirichlet polynomials (see \cite{Wat1}), for which we would need a character analogue in our setting. A character version of Watt's mean value theorem is indeed available thanks to the work in \cite{Wat2}. It was actually used in \cite{HarCarmII}, but this character version is not an exact analogue of Watt's mean value theorem for Dirichlet polynomials, as pointed out in the \cite[Remark on page 643]{HarCarmI}. This makes it difficult to push the exponent to the current limit of technology. It would be interesting if the said character version of Watt's mean value theorem can be improved in the case of prime power moduli. 

We point out that our approach does not work for general moduli but is specifically suited for prime power moduli. This restriction will be explained later when we deal with the type I sums. An extension of our results to power-full moduli $q$ whose radical is less or equal $\mathcal{L}^C$ should be possible, but we confine ourselves to prime power moduli for simplicity.  

The original motivation for this work was a $p$-adic analogue to the question of the distribution of $l\alpha$ modulo 1, where $\alpha$ is an irrational real number and $l$ runs over the primes. We will explain this connection in section \ref{connection}. 

We make a last remark on Theorem \ref{mainresult}, our main result in this article: Using Harman's lower bound sieve from \cite{Har}, it should be possible to derive a version of this theorem with asymptotic estimates for primes in arithmetic progressions replaced by lower bounds and the exponent $1/4$ by the larger value $7/22$. In this situation of lower bounds, a further improvement  (in the best case to the above-mentioned exponent 0.4736) could be possible using the method in [14], at the cost of a larger exceptional set of cardinality $O(\exp((\log x)^{\delta}))$ for some $\delta>0$ (this is still small compared to what the original Bombieri-Vinogradov theorem gives). Of course, for this result to be non-trivial, we would need to be able to establish our result for a set of moduli whose cardinality exceeds $\exp((\log x)^{\delta})$, which is not the case in the present paper. Indeed, our method is particularly well-suited for small sets of moduli.  

Our paper is organized as follows. In section 3, we provide the necessary tools. In section 4, we prove our main result, Theorem \ref{mainresult}. A connection to a $p$-adic problem in Diophantine approximation, which originally motivated this work, will be described in section 5. In section 6, the appendix, we give a full proof of an average version of Harman's asymptotic sieve.   

\section{Notations and preliminaries}
To generate a set of primes which we can sieve in, we use the Siegel-Walfisz theorem (see \cite[page 114-116]{Bru}, for example). 

\begin{proposition}[Siegel-Walfisz] 
Let $C>0$ be given. Then there exists a constant $D=D(C)>0$ such that for all $q\le \mathcal{L}^C$ and $a$ with $(a,q)=1$, we have
$$
\pi(x;q,a)=\frac{1}{\varphi(q)}\int\limits_{2}^x \frac{dt}{\log t} +
O\left(xe^{-D\sqrt{\log x}}\right).
$$
\end{proposition}

More generally, we have the following which can be established along the same lines. 

\begin{proposition} \label{SW}
Let $C>0$ be given. Then there exists a constant $D=D(C)>0$ such that for all $y\le x$, $q\le \mathcal{L}^C$ and $a$ with $(a,q)=1$, we have
$$
\pi(y;q,a)=\frac{1}{\varphi(q)}\int\limits_{2}^y \frac{dt}{\log t} +
O\left(xe^{-D\sqrt{\log x}}\right).
$$
\end{proposition}

We shall apply Proposition \ref{SW} to generate primes in a residue class $d_p \bmod{p}$, where we assume that
\begin{equation} \label{SiegelWalfiszcond}
p \le \mathcal{L}^C
\end{equation}
for an arbitrary but fixed $C$.
Then we sieve for primes in a residue class $e_p \bmod{p^{N_p}}$ properly contained in it, i.e. with $N_p>1$ and $e_p\equiv d_p\bmod{p}$. Here we assume that $p^{N_p}\in (Q,2Q]$ and $(d_p,p)=1$ (and hence $(e_p,p^{N_p})=1$). For $y_p\le x$, we set  
$$
\mathcal{A}_p:=\left\{n\le y_p: n \equiv e_p \bmod{p^{N_p}}\right\},
$$
$$
\mathcal{B}_p:=\left\{n\le y_p: n \equiv d_p \bmod{p}\right\}.
$$
If $\mathcal{M}$ is a finite set of integers and $z\ge 1$, we use the notation
$$
S(\mathcal{M},z):=\sharp\{n\in \mathcal{M} : p|n \mbox{ prime } \Rightarrow p\ge z\},
$$
which is common in sieve theory. We note that
\begin{equation} \label{SpiA}
\pi(y_p;p^N,e_p)=S(\mathcal{A}_p,x^{1/2})+O\left(\frac{x^{1/2}}{p^{N_p}}+1\right)
\end{equation}
and 
\begin{equation} \label{SpiB}
\pi(y_p;p,d_p)=S(\mathcal{B}_p,x^{1/2})+O\left(\frac{x^{1/2}}{p}+1\right).
\end{equation}

For now, we fix $p$ and set $\mathcal{A}:=\mathcal{A}_p$ and $\mathcal{B}:=\mathcal{B}_p$.  
To deduce information on $S(\mathcal{A},x^{1/2})$ from information on $S(\mathcal{B},x^{1/2})$, we shall use Harman's sieve. Below is an asymptotic version of Harman's sieve (see \cite[Theorem 3.1.]{HarPri} with $R=1$). 
  
\begin{proposition}[Harman] \label{Harmanasymp} Suppose that for all sequences $(a_m)_{m\in \mathbb{N}}$ and $(b_n)_{n\in \mathbb{N}}$ with $|a_m|\le \tau(m)$ and $|b_n|\le \tau(n)$ {\rm (}here $\tau(n)$ denotes the number of divisors of $n${\rm )}, we have  
\begin{equation} \label{typeI}
\left| \sum\limits_{\substack{mn\in \mathcal{A}\\ m\le M}} a_m -  
\lambda\sum\limits_{\substack{mn\in \mathcal{B}\\ m\le M}} a_m\right| \le X
\end{equation}
and 
\begin{equation} \label{typeII}
\left|\sum\limits_{\substack{mn\in \mathcal{A}\\ x^{\alpha}< m\le x^{\alpha+\beta}}} a_mb_n- \lambda\sum\limits_{\substack{mn\in \mathcal{B}\\ x^{\alpha}< m\le x^{\alpha+\beta}}} a_mb_n\right| \le  X
\end{equation}
for some fixed $\lambda>0$, $0<\alpha<1$, $\beta\le 1/2$, $M>x^{\alpha}$ and $X\ge 1$. Then  
\begin{equation} \label{Sasymp}
\left|S\left(\mathcal{A},x^{\beta}\right)-\lambda S\left(\mathcal{B},x^{\beta}\right)\right|=O\left(X\mathcal{L}^3\right).
\end{equation}
\end{proposition}

The above sieve result is not quite sufficient for our purposes. We need an average version, where $p$ runs over a set $\mathcal{P}$ of primes less or equal $\mathcal{L}^C$ (and correspondingly, $p^{N_p}$ runs over a set $\mathcal{S}$). Looking at the proof of Harman's asymptotic sieve in \cite{Har}, we observe that it suffices to assume that the  type I and type II relations hold for certain {\it specific} sequences $(a_m)$ and $(b_n)$ which do not depend on the sets $\mathcal{A}$ and $\mathcal{B}$ but just on $x$. This allows us to extend Proposition \ref{Harmanasymp} to the following result whose proof is along the same lines with an additional use of the Cauchy-Schwarz inequality. 

\begin{proposition} \label{Harmanasymp1} Suppose that for all sequences $(a_m)_{m\in \mathbb{N}}$ and $(b_n)_{n\in \mathbb{N}}$ with $|a_m|\le \tau(m)$ and $|b_n|\le \tau(n)$, we have 
\begin{equation} \label{typeI1}
\sum\limits_{p\in \mathcal{P}} \left| \sum\limits_{\substack{mn\in \mathcal{A}_p\\ m\le M}} a_m -\lambda 
\sum\limits_{\substack{mn\in \mathcal{B}_p\\ m\le M}} a_m \right|^2 \le Y
\end{equation}
and
\begin{equation} \label{typeII1}
\sum\limits_{p\in \mathcal{P}} \left| \sum\limits_{\substack{mn\in \mathcal{A}_p\\ x^{\alpha}< m\le x^{\alpha+\beta}}} a_mb_n- \lambda\sum\limits_{\substack{mn\in \mathcal{B}_p\\ x^{\alpha}< m\le x^{\alpha+\beta}}} a_mb_n \right|^2 \le Y
\end{equation}
for some fixed $\lambda>0$, $0<\alpha<1$, $\beta\le 1/2$, $M>x^{\alpha}$ and $Y\ge 1$. Then 
\begin{equation} \label{Sasymp1}
\sum\limits_{p\in \mathcal{P}} \left| S\left(\mathcal{A}_p,x^{\beta}\right)-\lambda S\left(\mathcal{B}_p,x^{\beta}\right)\right|^2 = O\left(Y\mathcal{L}^6\right).
\end{equation}
\end{proposition}

We shall supply a complete proof of Proposition \ref{Harmanasymp1} in the appendix.  Following usual custom, we call the bilinear sums in \eqref{typeI1} type I sums
and the bilinear sums in \eqref{typeII1} type II sums. We shall obtain a satisfactory type I estimate by elementary means and a satisfactory type II estimate by applying a dispersion argument and then using the large sieve after detecting the implicit congruence relations using Dirichlet characters. Below is a version of the large sieve (see \cite[Satz 5.5.1.]{Bru}, for example).

\begin{proposition}[Large Sieve] \label{ls} Let $Q$ and $N$ be positive integers and $M$ be an integer. Then, we have
$$
\sum\limits_{q\le Q} \frac{q}{\varphi(q)} \sideset{}{^\ast}\sum\limits_{\chi \bmod q} \left| \sum\limits_{M<n\le M+N} a_n\chi(n)\right|^2 \le \left(Q^2+N-1\right) \sum\limits_{M<n\le M+N} |a_n|^2,
$$
where the asterisk indicates that the sum is restricted to primitive characters. 
\end{proposition}

Another technical devise which we shall use to make ranges of variables independent is the following approximate version of Perron's formula (see \cite[Lemma 1.4.2.]{Bru}). 

\begin{proposition}[Perron's formula] \label{Perron} Let $c>0$, $Z\ge 2$ and $T\ge 2$. Let $(c_n)_{n\in \mathbb{N}}$ be a sequence of complex numbers and assume that the corresponding Dirichlet series $\sum\limits_{n=1}^{\infty} c_nn^{-s}$ converges absolutely for $s=c$. Then
$$
\sum\limits_{n\le Z} c_n=\frac{1}{2\pi i} \int\limits_{c-iT}^{c+iT} \left(\sum\limits_{n=1}^{\infty} c_nn^{-s}\right) Z^s \frac{ds}{s}+
O\left(\frac{Z^c}{T} \sum\limits_{n=1}^{\infty} |c_n|n^{-c} + C_Z\left(1+\frac{Z\log Z}{T}\right)\right),
$$
where 
$$
C_Z:=\max\limits_{3Z/4\le n\le 5Z/4} |c_n|. 
$$
\end{proposition}

\section{Proof of Theorem \ref{mainresult}}
To prove Theorem \ref{mainresult}, we apply Proposition \ref{Harmanasymp1} with $\beta=1/2$ and $M=x^{\alpha}+1$, where the parameter $\alpha$ will later be optimized. We need to establish \eqref{typeI1} and \eqref{typeII1} with 
\begin{equation} \label{choice}
\lambda:=\frac{1}{p^{N-1}} \quad \mbox{and} \quad Y:=\frac{x^2\mathcal{L}^{10}}{Q^2},
\end{equation}
where we recall that $p^{N_p}\in (Q,2Q]$ for all $p\in \mathcal{P}$. 
Then \eqref{Sasymp1} together with \eqref{SpiA}, \eqref{SpiB} and Proposition \ref{SW} implies 
\begin{equation*}
\sum\limits_{p\in \mathcal{P}} \left|F(y_p;p^{N_p},e_p)\right|^2= O\left(\frac{x^2\mathcal{L}^{16}}{Q^2}\right) 
\end{equation*}
for $Q$ as in Theorem \ref{mainresult}. Hence, the number of primes $p$ in $\mathcal{P}$ for which
$$
|F(y_p;p^{N_p},e_p)|> \frac{x}{Q\mathcal{L}^A}
$$ 
is bounded by $O(\mathcal{L}^{16+2A})$. The statement of Theorem \ref{mainresult} follows upon choosing $e_p$ and $y_p$ in such a way that
$$
|F(y_p;p^{N_p},e_p)|=F(x,p^{N_p}).
$$ 
  
\subsection{Treatment of type I sums}
In the following, we suppress the indices $p$ at $\mathcal{A}_p$, $\mathcal{B}_p$, $d_p$, $e_p$, $y_p$ and $N_p$ for simplicity. 
The difference of double sums over $m$ and $n$ in \eqref{typeI1} then equals
\begin{equation*} 
\Sigma_{I}=\sum\limits_{\substack{m\le M\\ (m,p)=1}} a_m\left(\sum\limits_{\substack{n\\ mn\in A}} 1 - \lambda \sum\limits_{\substack{n\\ mn\in B}} 1 \right)
\end{equation*}
which in our setting (recall the choice of $\lambda$ in \eqref{choice}) takes the form
\begin{equation*}
\Sigma_{I}=\sum\limits_{\substack{m\le M\\ (m,p)=1}} a_m\left(\sum\limits_{\substack{n\le y/m\\ n \equiv e\overline{m} \bmod{p^N}}} 1 - \frac{1}{p^{N-1}} \sum\limits_{\substack{n\le y/m\\ n \equiv d\overline{m} \bmod{p}}} 1 \right), 
\end{equation*}
where $\overline{m}$ denotes a multiplicative inverse of $m$ modulo $p^N$. We observe that the difference contained in the sum on the right-hand side above is $O(1)$, and hence we have 
$$
\Sigma_{I}\ll  \sum\limits_{m\le M} a_m \ll \sum\limits_{m\le M} \tau(m) \ll M\mathcal{L}, 
$$
which gives a bound of $O(M^2\mathcal{L}^2\sharp\mathcal{P})$ for the left hand side of \eqref{typeI1}.  Hence, \eqref{typeI1} holds with $Y$ as defined in \eqref{choice} if
\begin{equation*}
M\ll \frac{x\mathcal{L}^{4}}{Q(\sharp \mathcal{P})^{1/2}}. 
\end{equation*}
By the prime number theorem we have 
\begin{equation} \label{Pbound}
\sharp\mathcal{P} \ll \frac{\mathcal{L}^C}{\log \mathcal{L}}\ll x^{\varepsilon}.
\end{equation}
Hence, recalling our choice $M=x^{\alpha}+1$, we conclude that \eqref{typeI1} holds under the condition 
\begin{equation} \label{typeIcond}
\boxed{Q\le x^{1-\alpha-\varepsilon}.}
\end{equation}

{\bf Remark:} Here we see why our method is limited to prime power moduli: Assume we take general moduli $q_1$ in place of $p$ and $q_2$ in place of $p^N$, where $q_1|q_2$. Assume further that $(d,q_1)=1$, $(e,q_2)=1$ and $e \equiv d \bmod{q_1}$. Then for the above cancellation argument to work, the difference between the double sums in \eqref{typeI1} needs to be replaced by
$$
\sum\limits_{\substack{m\le M\\ n\le y/m\\ mn\equiv e \bmod{q_2}\\}} a_m- \frac{q_1}{q_2} 
\sum\limits_{\substack{m\le M\\ n\le y/m\\ mn\equiv d \bmod{q_1}\\ }} a_m 
$$
which equals
$$
\sum\limits_{\substack{m\le M\\ (m,q_2)=1}} a_m \left(\sum\limits_{\substack{n\le y/m\\ n \equiv e\overline{m} \bmod{q_2}}} 1 - \frac{q_1}{q_2} \sum\limits_{\substack{n\le y/m\\ n \equiv d\overline{m} \bmod{q_1}}} 1 \right).
$$
However, in general, the factor $q_1/q_2$ does not match with the factor $\lambda=\varphi(q_1)/\varphi(q_2)$ which we need in \eqref{Sasymp1} since we expect
$$
\pi(y;q_2,e)\sim \frac{\varphi(q_1)}{\varphi(q_2)}\cdot \pi(y;q_1,d).
$$  
For $q_1=p$ and $q_2=p^N$, though, both factors above are equal. A similar issue comes up in the treatment of the type II sums. 

In section 2, we mentioned that an extension to power-full moduli with radical less or equal $\mathcal{L}^C$ should be possible. Indeed, if $q_2$ is power-full, i.e. all exponents in its prime factorisation are greater than 1, and $q_1$ is the radical of $q_2$, then we have $q_1/q_2=\varphi(q_1)/\varphi(q_2)$. 

\subsection{Treatment of type II sums.}
Suppressing again the index $p$, the difference of double sums in \eqref{typeII1} equals  
\begin{equation} \label{introSigma} 
\Sigma_{II}:=\sum\limits_{\substack{x^{\alpha}<m\le x^{\alpha+\beta}}} a_m \left(\sum\limits_{\substack{n\\ mn\in \mathcal{A}}} b_n -  \lambda\sum\limits_{\substack{n\\ mn\in \mathcal{B}}} b_n\right).
\end{equation}
In our setting,
\begin{equation} \label{Sigmadef} 
\Sigma_{II}=\sum\limits_{\substack{x^{\alpha}<m\le x^{\alpha+\beta}\\ (m,p)=1}} a_m \left(\sum\limits_{\substack{n\le y/m\\ n\equiv e\overline{m} \bmod{p^N}}} b_n -  \frac{1}{p^{N-1}}\sum\limits_{\substack{n\le y/m\\ n\equiv d\overline{m} \bmod{p}}} b_n\right).
\end{equation}
We split $\Sigma_{II}$ into $O(\mathcal{L})$ sub-sums of the form
\begin{equation} \label{aftersplit}
\Sigma(K):=\sum\limits_{\substack{K<m\le K'\\ (m,p)=1}} a_m \left(\sum\limits_{\substack{n\le y/m\\ n\equiv e\overline{m} \bmod{p^N}}} b_n -  \frac{1}{p^{N-1}}\sum\limits_{\substack{n\le y/m\\ n\equiv d\overline{m} \bmod{p}}} b_n\right)
\end{equation}
with $x^{\alpha}\le K<K'\le 2K\le x^{\alpha+\beta}$. Throughout the following, let
$$
L:=\frac{x}{K}.
$$
To disentangle the summation variables, we apply Perron's formula, Proposition \ref{Perron}, with
$$
Z:=\frac{y}{m}, \quad c:=\frac{1}{\log L}, \quad T:=L\log L
$$
and 
$$
c_n:= \begin{cases} b_n & \mbox{ if } n\le L \mbox{ and } n \equiv e \overline{m} \bmod{p^N} 
\mbox{ (or } n \equiv d\overline{m} \bmod{p})\\ 0 & \mbox{ otherwise} \end{cases} 
$$ 
to the inner sums over $n$ on the right-hand side of \eqref{aftersplit}. This gives
\begin{equation*}
\begin{split} 
& \sum\limits_{\substack{n\le y/m\\ n\equiv e\overline{m} \bmod{p^N}}} b_n -  \frac{1}{p^{N-1}}\sum\limits_{\substack{n\le y/m\\ n\equiv d\overline{m} \bmod{p}}} b_n\\
= & \frac{1}{2\pi i} \int\limits_{c-iT}^{c+iT} \left(\sum\limits_{\substack{n\le L\\ n\equiv e\overline{m} \bmod{p^N}}} b_n n^{-s} -  \frac{1}{p^{N-1}}\sum\limits_{\substack{n\le L\\ n\equiv d\overline{m} \bmod{p}}} b_n n^{-s} \right) \left(\frac{y}{m}\right)^{s} \frac{ds}{s} + O\left(x^{\varepsilon}\right).
\end{split}
\end{equation*}
Hence, we have
\begin{equation*} 
\Sigma(K)=\frac{1}{2\pi i} \int\limits_{c-iT}^{c+iT} \Sigma(K,s)\frac{ds}{s} + O(Kx^{\varepsilon}),
\end{equation*}
where 
\begin{equation*} 
\Sigma(K,s):=\sum\limits_{\substack{K<m\le K'\\ (m,p)=1}} a_m(s) \left(\sum\limits_{\substack{n\le L\\ n\equiv e\overline{m} \bmod{p^N}}} b_n(s) -  \frac{1}{p^{N-1}}\sum\limits_{\substack{n\le L\\ n\equiv d\overline{m} \bmod{p}}} b_n(s)\right)
\end{equation*}
with 
$$
a_m(s):=a_m\left(\frac{y}{m}\right)^s, \quad b_n(s):=b_nn^{-s}.
$$
We note that $|a_m(s)|\ll |a_m|\le \tau(m)$ if $K< m\le K'$ and $|b_n(s)|\ll |b_n|\le \tau(n)$ if $n\le L$. By the Cauchy-Schwarz inequality for integrals, we deduce that
\begin{equation} \label{aftercauchy1}
\begin{split}
|\Sigma(K)|^2 \ll & \left(\int\limits_{-T}^{T} \frac{dt}{|c+it|}\right) \cdot \int\limits_{-T}^T \frac{|\Sigma(K,c+it)|^2}{|c+it|}dt + K^2x^{2\varepsilon} \\
\ll &
\mathcal{L} \cdot \int\limits_{-T}^T \frac{|\Sigma(K,c+it)|^2}{|c+it|}dt + K^2x^{2\varepsilon}. 
\end{split}
\end{equation}
Set $s=c+it$. An application of the Cauchy-Schwarz inequality for sums gives
\begin{equation} \label{CS} 
|\Sigma(K,s)|^2\le \left(\sum\limits_{\substack{K<m\le K'\\ (m,p)=1}} |a_m(s)|^2\right) \cdot \Sigma'(K,s)
\le K\mathcal{L}^3 \cdot \Sigma'(K,s),
\end{equation}
where 
\begin{equation*}
\Sigma'(K,s):=\sum\limits_{\substack{K<m\le 2K\\ (m,p)=1}} \left|\sum\limits_{\substack{n\le L\\ n\equiv e\overline{m} \bmod{p^N}}} b_n(s) -  \frac{1}{p^{N-1}}\sum\limits_{\substack{n\le L\\ n\equiv d\overline{m} \bmod{p}}} b_n(s)\right|^2.
\end{equation*}

Now we use a dispersion argument. We multiply out the modulus square and re-arrange summations to get
\begin{equation*}
\Sigma'(K,s)=\Sigma_1(K,s)-\Sigma_2(K,s)-\Sigma_3(K,s)+\Sigma_4(K,s),
\end{equation*}
where
$$
\Sigma_1(K,s):=\sum\limits_{\substack{n_1,n_2\le L\\ n_1 \equiv n_2 \bmod{p^N}\\ (n_1n_2,p)=1}} b_{n_1}(s)\overline{b_{n_2}(s)} 
\sum\limits_{\substack{K<m\le 2K\\ m \equiv e\overline{n_1} \bmod{p^N}}} 1, 
$$
$$
\Sigma_2(K,s):=\frac{1}{p^{N-1}} \sum\limits_{\substack{n_1,n_2\le L\\ n_1 \equiv n_2 \bmod{p}\\ (n_1n_2,p)=1}} b_{n_1}(s)\overline{b_{n_2}(s)} 
\sum\limits_{\substack{K<m\le 2K\\ m \equiv e\overline{n_1} \bmod{p^N}}} 1,
$$
$$
\Sigma_3(K,s):=\frac{1}{p^{N-1}} \sum\limits_{\substack{n_1,n_2\le L\\ n_1 \equiv n_2 \bmod{p}\\ (n_1n_2,p)=1}} b_{n_1}(s)\overline{b_{n_2}(s)} 
\sum\limits_{\substack{K<m\le 2K\\ m \equiv e\overline{n_2} \bmod{p^N}}} 1,
$$
and 
$$
\Sigma_4(K,s):=\frac{1}{p^{2N-2}} \sum\limits_{\substack{n_1,n_2\le L\\ n_1 \equiv n_2 \bmod{p}\\ (n_1n_2,p)=1}} b_{n_1}(s)\overline{b_{n_2}(s)} 
\sum\limits_{\substack{K<m\le 2K\\ m \equiv d\overline{n_1} \bmod{p}}} 1.
$$
We see immediately that
$$
\Sigma_1(K,s)=\sum\limits_{\substack{n_1,n_2\le L\\ n_1 \equiv n_2 \bmod{p^N}\\ (n_1n_2,p)=1}} b_{n_1}(s)\overline{b_{n_2}(s)} \left(\frac{K}{p^N}+O(1)\right),
$$
$$
\Sigma_2(K,s)=\frac{1}{p^{N-1}} \sum\limits_{\substack{n_1,n_2\le L\\ n_1 \equiv n_2 \bmod{p}\\ (n_1n_2,p)=1}} b_{n_1}(s)\overline{b_{n_2}(s)} \left(\frac{K}{p^N}+O(1)\right),
$$
$$
\Sigma_3(K,s)=\frac{1}{p^{N-1}} \sum\limits_{\substack{n_1,n_2\le L\\ n_1 \equiv n_2 \bmod{p}\\ (n_1n_2,p)=1}} b_{n_1}(s)\overline{b_{n_2}(s)} 
\left(\frac{K}{p^N}+O(1)\right),
$$
and 
$$
\Sigma_4(K,s)=\frac{1}{p^{2N-2}} \sum\limits_{\substack{n_1,n_2\le L\\ n_1 \equiv n_2 \bmod{p}\\ (n_1n_2,p)=1}} b_{n_1}(s)\overline{b_{n_2}(s)} 
\left(\frac{K}{p}+O(1)\right)
$$
so that 
\begin{equation*}
\begin{split}
\Sigma'(K,s)=& \frac{K}{p^{N}}\sum\limits_{\substack{n_1,n_2\le L\\ n_1 \equiv n_2 \bmod{p^N}\\ (n_1n_2,p)=1}} b_{n_1}(s)\overline{b_{n_2}(s)} -\frac{K}{p^{2N-1}}\sum\limits_{\substack{n_1,n_2\le L\\ n_1 \equiv n_2 \bmod{p}\\ (n_1n_2,p)=1}} b_{n_1}(s)\overline{b_{n_2}(s)}\\
&  +O\left(\frac{L^2x^{\varepsilon}}{p^N}\right),
\end{split}
\end{equation*}
provided that $p^N\ll L$ for all $L$ in question (i.e. for $L= x/K\ge x^{1-(\alpha+\beta)}$), which is the case if
\begin{equation}\label{somecond}
\boxed{
Q \le x^{1-(\alpha+\beta)}.}
\end{equation}
Here we recall that $p^N\in (Q,2Q]$. 

Now we use Dirichlet characters to detect the 
congruence relations in the above sums. We thus get 
\begin{equation*} \label{Sigma''}
\begin{split}
\Sigma'(K,s) = & \frac{K}{p^{N}} \cdot \frac{1}{\varphi(p^N)}\sum\limits_{\chi \bmod p^N} \sum\limits_{n_1,n_2\le L} b_{n_1}(s)\overline{b_{n_2}(s)} \chi(n_1)\overline{\chi}(n_2)-\\ 
&
\frac{K}{p^{2N-1}}\cdot \frac{1}{\varphi(p)} \sum\limits_{\chi' \bmod{p}}\ \sum\limits_{n_1,n_2\le L} b_{n_1}(s)\overline{b_{n_2}(s)} \chi'(n_1)\overline{\chi'}(n_2)+O\left(\frac{L^2x^{\varepsilon}}{p^N}\right).
\end{split}
\end{equation*}
Note that 
$$
\frac{K}{p^N}\cdot \frac{1}{\varphi(p^N)}=\frac{K}{p^{2N-1}}\cdot \frac{1}{\varphi(p)}=\frac{K}{\varphi(p^{2N})}
$$
and hence 
\begin{equation*}
\begin{split}
\Sigma'(K,s) = & \frac{K}{\varphi(p^{2N})} \sum\limits_{\chi \in \mathcal{X}(p^N)} \sum\limits_{n_1,n_2\le L} b_{n_1}(s)\overline{b_{n_2}(s)}\chi(n_1)\overline{\chi}(n_2) +O\left(\frac{L^2x^{\varepsilon}}{p^N}\right)\\
= & \frac{K}{\varphi(p^{2N})}\sum\limits_{\chi\in \mathcal{X}(p^N)} \left|\sum\limits_{n\le L} b_n(s) \chi(n)
\right|^2+ O\left(\frac{L^2x^{\varepsilon}}{p^N}\right),
\end{split}
\end{equation*}
where $\mathcal{X}(p^N)$ is the set of all Dirichlet characters modulo $p^N$ which are not induced by a Dirichlet character modulo $p$ (in particular, $\mathcal{X}(p^N)$ does not contain the principal character). We may write the above as
\begin{equation} \label{Sigma''}
\Sigma'(K,s) = \frac{K}{\varphi(p^{2N})} \sum\limits_{h=2}^N\ \sideset{}{^{\ast}} \sum\limits_{\chi \bmod{p^h}} \left|  \sum\limits_{n\le L} b_n(s) \chi(n)
\right|^2+ O\left(\frac{L^2x^{\varepsilon}}{p^N}\right),
\end{equation}
where the asterisk indicates that $\chi$ ranges over all {\it primitive} characters modulo $p^h$. 

Now we apply the large sieve, Proposition \ref{ls}, after re-introducing the indices $p$ and summing over $p\in \mathcal{P}$. This gives us
\begin{equation} \label{Largesieve}
\begin{split}
\sum\limits_{p\in \mathcal{P}} \sum\limits_{h=2}^{N_p} \frac{p^h}{\varphi(p^h)}\ \sideset{}{^{\ast}} \sum\limits_{\chi \bmod{p^h}}\left| \sideset{}{^{\ast}}\sum\limits_{n\le L} b_n(s) \chi(n)
\right|^2 \ll & \left(Q^2+L\right)\sum\limits_{n\le L}
|b_n|^2\\
 \ll & \left(Q^2+L\right)L\mathcal{L}^3,
\end{split}
\end{equation}
where we recall that $p^{N_p}\in (Q,2Q]$. 
Noting that 
$$
\frac{K}{\varphi(p^{2N})}=\frac{K}{p^{2N}}\cdot \frac{p^h}{\varphi(p^h)}
$$ 
for $h\ge 2$, and recalling that $KL=x$, we deduce that
$$
\sum\limits_{p\in \mathcal{P}} |\Sigma_p'(K,s)|\ll Q^{-2}K^{-1}x^2\mathcal{L}^3 + x\mathcal{L}^3+Q^{-1}K^{-2}x^{2+\varepsilon}\sharp\mathcal{P}
$$
after re-introducing the index $p$. 
Combining this with \eqref{Pbound}, \eqref{aftercauchy1} and \eqref{CS}, we obtain
\begin{equation} \label{bound1}
\sum\limits_{p\in \mathcal{P}} |\Sigma_p(K)|^2 \ll Q^{-2}x^2\mathcal{L}^8+Kx\mathcal{L}^8 +Q^{-1}K^{-1}x^{2+3\varepsilon} + K^2x^{3\varepsilon}.
\end{equation}
We get a second bound for the left-hand side by reversal of roles of the variables $m$ and $n$ in the above process, where $K$ on the right-hand side is replaced by $x/K$, i.e.
\begin{equation} \label{bound2}
\sum\limits_{p\in \mathcal{P}} |\Sigma_p(K)|^2 \ll Q^{-2}x^2\mathcal{L}^{8} + K^{-1}x^2\mathcal{L}^{8}+ Q^{-1}Kx^{1+3\varepsilon}+K^{-2}x^{2+3\varepsilon}. 
\end{equation}
For this to hold, the condition \eqref{somecond} needs to be replaced by
\begin{equation}\label{somecond2}
\boxed{
Q \le x^{\alpha}.}
\end{equation}
Using \eqref{bound1} if $K\le x^{1/2}$ and \eqref{bound2} if $K\ge x^{1/2}$, and recalling that 
$x^{\alpha}\le K\le x^{\alpha+\beta}$, we deduce that
\begin{equation*}
\sum\limits_{p\in \mathcal{P}} |\Sigma_p(K)|^2 \ll Q^{-2}x^2\mathcal{L}^{8} + x^{3/2}\mathcal{L}^{8}+ Q^{-1}x^{2-\alpha+3\varepsilon}\\
+Q^{-1}x^{1+\alpha+\beta+3\varepsilon}.
\end{equation*}
With the choice 
$$
\alpha=\frac{1}{4}\quad  \mbox{and} \quad  \beta=\frac{1}{2},
$$ 
and writing $\Sigma_p=\Sigma_{II}$ with $\Sigma_{II}$ as in \eqref{Sigmadef},
we therefore obtain
\begin{equation}
\begin{split}
\sum\limits_{p\in \mathcal{P}} |\Sigma_p|^2\ll & Q^{-2}x^2\mathcal{L}^{10} + x^{3/2}\mathcal{L}^{10} + Q^{-1}x^{7/4+4\varepsilon}
\end{split}
\end{equation}
using the Cauchy-Schwarz inequality again.  

The right-hand side is bounded by $Y$ defined in \eqref{choice} if 
\begin{equation} \label{lastcond}
\boxed{ Q\le x^{1/4-4\varepsilon}.}
\end{equation} 

\subsection{Completion of the proof} In view of our conditions \eqref{typeIcond}, \eqref{somecond}, \eqref{somecond2} and \eqref{lastcond}, Theorem \ref{mainresult} follows.   

\section{Connection to $p$-adic Diophantine approximation with primes $l$} \label{connection}
In this section, we describe the problem which actually motivated us originally.

Let $\theta\in \mathbb{R}$ be irrational. The problem of how $l\theta$ is distributed modulo one as $l$ runs over the primes has received a lot of attention. Currently, the best known result in this direction is due to Matom\"aki \cite{Mat} who proved that for any $\varepsilon>0$, there are infinitely many primes $l$ such that 
$$
||l\theta||\le l^{-1/3+\varepsilon},
$$
where $||.||$ denotes the distance to the closest integer. A brief description of the history of this problem is given in \cite[section 1]{BT}. In the context of quadratic number fields, this problem has been considered in \cite{HarQi}, \cite{BM} and \cite{BT}. 
It is natural to study this problem also in the setting of $p$-adic numbers. Fix $\alpha\in U$, where 
$$
U=\mathbb{Z}_p\setminus p\mathbb{Z}_p
$$ is the set of units in the ring of $p$-adic integers $\mathbb{Z}_p$. For $z\in \mathbb{Z}_p$ given by
$$
z=c_0+c_1p+c_2p^2+\cdots
$$
define $||z||_p$ as 
$$
||z||_p:=|c_1p+c_2p^2+\cdots|_p=p^{-n},
$$  
where $n$ is the smallest positive integer such that $c_n\not=0$. We aim to find $\tau>0$ such that there exist infinitely many rational primes $l$ such that
$$
||l\alpha||_p<l^{-\tau}.
$$
In the following, we relate this problem to primes in arithmetic progressions with prime power moduli. 

Let
$$
\alpha=a_0+a_1p+a_2p^2+... \quad (a_0\not=0, \ 0\le a_m<p)
$$
and write 
$$
l=b_0+b_1p+...+b_np^n \quad (l\not=p, \ 0\le b_m<p)
$$
in $p$-adic expansion.
Then 
$$
||l\alpha||_p< p^{-M} \quad (M\in \mathbb{N})
$$
if and only if the following congruences hold:
\begin{equation} \label{congru1}
\begin{split}
b_0a_1+b_1a_0\equiv & 0 \bmod{p}\\
d_1p^{-1}+b_0a_2+b_1a_1+b_2a_0 \equiv & 0 \bmod{p}\\
\vdots & \\
d_{M-1}p^{-(M-1)}+ b_0a_{M}+b_1a_{M-1}+\cdots +b_{M}a_0\equiv & 0 \bmod{p},
\end{split}
\end{equation}
where 
$$
d_n:=(a_0b_1+a_1b_0)+(a_0b_2+a_1b_1+a_2b_0)p+\cdots+(a_0b_{n}+a_1b_{n-1}+\cdots +a_{n}b_0)p^{n-1}.
$$
These are equivalent to
\begin{equation} \label{congru2}
\begin{split}
b_1\equiv & -b_0a_1\overline{a_0} \bmod{p}\\
b_2\equiv & - (d_1p^{-1}+b_0a_2+b_1a_1)\overline{a_0} \bmod{p}\\
\vdots & \\
b_{M} \equiv &- (d_{M-1}p^{-(M-1)}+b_0a_{M}+b_1a_{M-1}+\cdots +b_{M-1}a_1)\overline{a_0}\bmod{p},
\end{split}
\end{equation}
where $\overline{a_0}$ denotes a multiplicative inverse of $a_0$ modulo $p$. It is reasonable to fix $b_0$, i.e. to restrict $l$ from the beginning to a residue class $l\equiv b_0\bmod{p}$. Then $b_1,b_2,...,b_{M}$ are uniquely determined by the above congruences. 
In particular, if we assume that $a_0b_0\equiv 1 \bmod{p}$, then the above system reduces to the single congruence
$$
(a_0+a_1p+\cdots +a_Mp^M)(b_0+b_1p+\cdots +b_Mp^M)\equiv 1 \bmod{p^{M+1}}.
$$

Now our problem may be re-formulated as follows:
Find $\tau>0$ such that if $K$ is large enough, then there exists a rational prime $l\le p^K$ whose first $\lceil \tau K\rceil+1$ digits in $p$-adic expansion are equal to $b_0,b_1,...,b_{\lceil\tau K\rceil}$. (Hence, the problem is now about primes $l$ with pre-assigned digits.) This in turn is just the same as demanding that
$$
l\equiv (b_0+b_1p+\cdots +b_{M}p^{M}) \bmod p^{M+1},
$$ 
where $M:=\lceil\tau K\rceil$. We are thus trying to prove that 
there exists a prime $l$ in the above residue class such that $l\le p^K$. Hence we are down to Linnik's problem on the least prime in a residue class modulo a power of $p$. The best known unconditional result on Linnik's problem for general moduli $q$ is due to Xylouris \cite{Xyl} who proved that if $q\in \mathbb{N}$ is large enough, then for every $a$ coprime to $q$, there exists a prime 
$$
l\equiv a\bmod{q}
$$  
such that 
$$
l\le q^{5}.
$$
However, for powers $q=p^N$ of a fixed odd prime $p$, the result by Banks and Shparlinski \cite{BaS} mentioned in section 1 allows us to replace the exponent $5$ by $1/0.4736=2.1115...$. 
Therefore, $\tau=0.4736$ is an admissible exponent in the $p$-adic $l\alpha$-problem and hence we have the following theorem.

\begin{theorem} \label{padicversion} Fix an odd prime $p$. Then there exists a positive number $K_0(p)$ such that the following holds. If $K\ge K_0(p)$, then there exists a rational prime $l\le p^K$ such that
$$
||l\alpha||_p< p^{-\lceil\tau K\rceil}\le l^{-\tau}
$$ 
for $\tau=0.4736$. 
\end{theorem}

Under GRH, the least prime $l \equiv a\bmod{q}$ is less than $q^{2+\varepsilon}$, and we can then replace the exponent $0.4736$ in Theorem \ref{padicversion} by $1/2-\varepsilon$. 

It is natural to try applying Harman's sieve to the $l\alpha$-problem in $p$-adic setting as it was done with success in the settings of $\mathbb{Q}$ and quadratic number fields (see \cite{BM}, \cite{BT}, \cite{HarQi} and \cite{Mat}). As outlined in section 1, if we had a full $p$-adic analogue of the method used in \cite{BHP}, we could improve the exponent $0.4736$ in Theorem \ref{padicversion} to $19/40=0.475$. 

\section{Appendix: Proof of Proposition \ref{Harmanasymp1}}
Our proof of Proposition \ref{Harmanasymp1} is along similar lines as the proof of \cite[Theorem 3.1]{HarPri}, and our exposition follows closely \cite[section 7]{BT} and \cite[section 13]{BM}, where weighted versions for quadratic number fields were proved.
We shall need the following lemma which Harman terms ``cosmetic surgery'' (see \cite[Lemma 2.2]{HarPri}).

\begin{lemma}  \label{CoSu}
For any two distinct real numbers $\rho,\gamma >0 $ and $T\geq1$ one has 
$$
   \Bigg|1_{\gamma<\rho}-\frac{1}{\pi}\int\limits_{-T}^{T}e^{i\gamma t} \frac{\sin(\rho t)}{t}\Bigg| \ll \frac{1}{T|\gamma-\rho|} ,
$$
where the implied constant is absolute.
\end{lemma}

We shall use the following notations throughout our proof.

\begin{itemize}
\item  For a general condition $(C)$, we write
$$
1_{\{(C)\}} :=\begin{cases} 1 & \mbox{ if } (C) \mbox{ is satisfied,}\\ 0 & \mbox{ if } (C) \mbox{ is not satisfied.}
\end{cases}
$$
\item If $M$ is a set, we write
$$
1_{M}(x) := \begin{cases} 1 \mbox{ if } x\in M,\\ 0 \mbox{ otherwise.} \end{cases}
$$
\item If $z>0$, we write
$$
P(z):=\prod\limits_{\substack{p<z \\ p \mbox{\ \rm \scriptsize prime}}} p 
$$
and
$$
\mathbb{P}(z):=\{p<z : p \mbox{ prime}\}.
$$
\item If $w : \mathbb{N} \rightarrow \mathbb{R}_{\ge 0}$ is a function such that $\sum\limits_{n=1}^{\infty} w(n)$ converges, then we write 
$$
S(w,z):= \sum\limits_{\substack{a\in \mathbb{N}\\ (a,P(z))=1}} w(a).
$$
\end{itemize} $ $\\
Throughout the following, let $z:=x^{\beta}$. By the above notations, 
$$
S(1_{\mathcal{M}},z)=S(\mathcal{M},z)
$$
for any finite set $\mathcal{M}$ of positive integers. Our proof begins by observing that
\begin{equation} \label{4.4}
\begin{split}
    S(w,z) =\sum\limits_{b\in \mathbb{N}} w(b)\sum\limits_{\substack{d|P(z)\\ d|b}}\mu(d)
    =\sum\limits_{d|P(z)} \mu(d)\sum\limits_{a\in \mathbb{N}} w(ad)
\end{split}
\end{equation}
for any function $w:\mathbb{N}\rightarrow \mathbb{R}_{\ge 0}$ function such that $\sum\limits_{n=1}^{\infty} w(n)$ converges. In the following, let $p\in \mathcal{P}$ and $\omega_p:=1_{\mathcal{A}_p}$ and $\tilde{\omega}_p:=\lambda\cdot 1_{\mathcal{B}_p}$. Set 
\begin{equation} \label{4.5}
\Delta_p(d)=\sum\limits_{a\in \mathbb{N}}(\omega_p(ad)-\tilde{\omega}_p(ad)).
\end{equation}
Applying \eqref{4.4} for $w=\omega_p$ and $w=\tilde{\omega}_p$ yields
\begin{equation} \label{4.6}
\begin{split}
S(\omega_p,z)-S(\tilde{\omega}_p,z)=\sum\limits_{d|P(z)}\mu(d)\Delta_p(d)=&
\bigg\{\sum\limits_{\substack{d|P(z)\\ d<M}} + \sum\limits_{\substack{d|P(z)\\ d\ge M}}\bigg\}\mu(d)\Delta_p(d)\\
=& S_p^{\sharp} + S_p^{\flat}, \mbox{ say}.
\end{split}
\end{equation}
Using \eqref{typeI1} with $a_{d}= \mu(d)1_{d|P(z)}$, we deduce that 
$$
\sum\limits_{p\in \mathcal{P}} \left|S_p^{\sharp}\right|^2 \leq Y.
$$  
Therefore, to prove the proposition, it suffices to show that 
\begin{align} \label{4.7}
\sum\limits_{p\in \mathcal{P}} \left|S_p^{\flat}\right|^2 \ll Y\mathcal{L}^6. 
\end{align} 

The next step is to arrange $S_p^{\flat}$ into subsums according to the sizes of the prime factors in $d$ (where $d$ is the summation variable in \eqref{4.6}).
Take $ g:\mathbb{N} \rightarrow\mathbb{C}$ to be any  function. We may group the terms of the sum
$$
S=\sum\limits_{d|P(z)}\mu(d)g(d)
$$
according to the largest prime factor $p_1$  of $d$, getting the identity
\begin{equation} \label{4.8}
S = g(1)-\sum\limits_{p_1<z}\sum\limits_{d|P(p_1)} \mu(d)g(p_1d).
\end{equation}
Similarly, for the part $\sum\limits_{d|P(p_1)} \mu(d)g(p_1d)$, we have 
\begin{equation} \label{4.9}
    \sum\limits_{d|P(p_1)} \mu(d)g(p_1d)= g(p_1)-\sum\limits_{p_2<p_1}\sum\limits_{d|P(p_2)} \mu(d)g(p_1p_2d).
\end{equation}
Minding the innermost sum on the right-hand side above, it is obvious that the above identity can be iterated if so desired. To describe for which sub-sums iteration is beneficial, we decompose the set $\mathbb{P}(z)$ of primes less than $z$ into
\begin{equation*}
\begin{split}
    \mathbb{P}(z)=&\{p_1\in \mathbb{P}(z):p_1>x^{\alpha}\}\mathbin{\dot{\cup}}\{p_1\in \mathbb{P}(z): p_1\le x^{\alpha}\}\\
    =&\mathcal{P}_1\mathbin{\dot{\cup}}\mathcal{Q}_1,\ \mbox{say},
    \end{split}
\end{equation*}
and inductively for $s=2,3,...$, we define
\begin{equation*}
\begin{split}
  \mathcal{Q'}_s =& \{(p_1,...,p_{s-1},p_s)\in\mathbb{P}(z)^s: p_s < p_{s-1},\ (p_1,...,p_{s-1})\in \mathcal{Q}_{s-1}\}\\
  =&\mathcal{P}_s\mathbin{\dot{\cup}}\mathcal{Q}_s,\ \mbox{say},
\end{split}
\end{equation*}
where
$$
  \mathcal{P}_s := \{(p_1,...,p_{s-1},p_s)
  \in\mathcal{Q'}_s: p_1p_2\cdots p_s>x^{\alpha}\}
$$
and 
$$
  \mathcal{Q}_s := \{(p_1,...,p_{s-1},p_s)\in\mathcal{Q'}_s: p_1p_2\cdots p_s\le x^{\alpha}\}.
$$

Assuming that $g$ vanishes on arguments $a$ with $a\le x^{\alpha}$, and on applying \eqref{4.8} and \eqref{4.9}, we have 
\begin{equation*}
\begin{split}
S= & -\bigg(\sum\limits_{p_1\in\mathcal{P}_1}+\sum\limits_{p_1\in\mathcal{Q}_1}\bigg)\sum\limits_{d|P(p_1)}\mu(d)g(p_1d)\\
  =& -\sum\limits_{p_1\in\mathcal{P}_1}\sum\limits_{d|P(p_1)}\mu(d)g(p_1d)+\sum\limits_{(p_1,p_2)\in \mathcal{P}_2}\sum\limits_{d|P(p_2)}\mu(d)g(p_1p_2d)\\ &
  + \sum\limits_{(p_1,p_2)\in \mathcal{Q}_2}\sum\limits_{d|P(p_2)}\mu(d)g(p_1p_2d).
\end{split}
\end{equation*}
On iterating this process - always applying \eqref{4.9} to the $\mathcal{Q}$-part - it transpires that 
\begin{equation*}
\begin{split}
   S= &\sum\limits_{s\leq t}(-1)^s\sum\limits_{(p_1,...,p_s)\in\mathcal{P}_s}\sum\limits_{d|P(p_s)}\mu(d)g(p_1p_2\cdots p_sd)\\
   & +(-1)^t\sum\limits_{(p_1,...,p_t)\in\mathcal{Q}_t}\sum\limits_{d|P(p_t)}\mu(d)g(p_1p_2\cdots p_td)
\end{split}
\end{equation*}
for any $t\in \mathbb{N}$. Since the product of $t$ primes is greater than or equal to $2^t$, we have
\begin{equation*}
    \mathcal{Q}_t=\emptyset \mbox{ for } t>\frac{\alpha}{\log 2}\cdot \log x.
\end{equation*}
Hence,
\begin{equation*}
    S=\sum\limits_{s\leq t}(-1)^s \sum\limits_{(p_1,...,p_s)\in\mathcal{P}_s}\sum\limits_{d|P(p_s)}\mu(d)g(p_1p_2\cdots p_sd)
\end{equation*}
for 
\begin{equation} \label{4.10}
    t:=\Bigl\lfloor{\frac{\alpha\log x}{\log 2}}\Bigr\rfloor+1 \ll \log x.   
\end{equation} 

We apply this to $S_p^{\flat}$ with $g(a)=\Delta_p(a)1_{\{a\ge M\}}$. Note that since $M>x^{\alpha}$, we have $g(a)=0$ for all $a\le x^{\alpha}$, as was assumed in the above arguments. Thus,
\begin{align} \label{4.11}
    S_p^{\flat}=\sum\limits_{s\leq t}  (-1)^sS_p^{\flat}(s),
\end{align}
where 
\begin{align*}
     S_p^{\flat}(s)=\sum\limits_{\substack{(p_1,...,p_s)\in\mathcal{P}_s \\ a=p_1\cdots p_s}}\sum\limits_{\substack{d|P(p_s)\\ ad\ge M}} \mu(d)\Delta_p(ad).
\end{align*}
Another application of \eqref{4.9} gives 
\begin{equation} \label{4.12}
\begin{split}
 S_p^{\flat}(s)=&\sum\limits_{\substack{(p_1,...,p_s)\in\mathcal{P}_s \\ a=p_1\cdots p_s\\ a\ge M}} \Delta_p(a)-\sum\limits_{\substack{(p_1,...,p_s)\in\mathcal{P}_s \\ a=p_1\cdots p_s}} \sum\limits_{p_0< p_s} \sum\limits_{\substack{d|P(p_0) \\ ap_0d\ge M}} \mu(d)\Delta_p(ap_0d) \\
 =& S_p^{\flat,1}(s)-S_p^{\flat,2}(s), \ \mbox{say}. 
\end{split}
\end{equation}
Using the Cauchy-Schwarz inequality and \eqref{4.10}, we therefore have
\begin{equation} \label{inter}
\sum\limits_{p\in \mathcal{P}} |S_p^{\flat}|^2 \ll \mathcal{L} \sum\limits_{j=1}^2\sum\limits_{s\le t} \sum\limits_{p\in \mathcal{P}} |S_p^{\flat,j}(s)|^2.
\end{equation}

Given $a=p_1\cdots p_{s-1}p_s$ with 
\begin{equation*}
(p_1,...,p_{s-1},p_s)\in \mathcal{P}_s \quad \mbox{and} \quad (p_1,...,p_{s-1})\in \mathcal{Q}_{s-1},
\end{equation*}
and noting that $p_s< p_1<z=x^{\beta}$, we have
\begin{align*}
  x^{\alpha} < a=p_1\cdots p_{s-1}p_s<x^{\alpha} \cdot x^{\beta}.
\end{align*}
Using this, we find that $S_p^{\flat,1}(s)$ can be expressed as 
\begin{align*}
 \mathop{\sum\sum}\limits_{a,b \in \mathbb{N}} c_a (\omega_p(ab)-\tilde{\omega}_p(ab)),
\end{align*}
where the coefficients
\begin{align*}
c_{a}=1_{\{a\geq M\}} 1_{\{p_1\cdots p_s: (p_1,...,p_s)\in \mathcal{P}_s\}}(a)
\end{align*}
are only supported on $a$ with $x^{\alpha}<a<x^{\alpha+\beta}$. Hence by \eqref{typeII1}, 
\begin{equation*} 
\sum\limits_{p\in \mathcal{P}} \left|S_p^{\flat,1}(s)\right|^2\leq Y. 
\end{equation*}
In view of \eqref{4.10} and \eqref{inter}, it therefore remains to establish that
\begin{equation} \label{4.13}
\sum\limits_{p\in \mathcal{P}} \left|S_p^{\flat,2}(s)\right|^2\ll Y\mathcal{L}^4
\end{equation}
for all $s$ less or equal $t$. 

Expanding the definition \eqref{4.5} of $\Delta_p$, we get
\begin{equation*}
    S_p^{\flat,2}(s)=S_p^{\flat,2}(s,\omega_p)-S_p^{\flat,2}(s,\tilde{\omega}_p),
\end{equation*}
where
\begin{equation*}
\begin{split}
 S_p^{\flat,2}(s,w):=&\sum\limits_{\substack{(p_1,...,p_s)\in\mathcal{P}_s \\ a=p_1\cdots p_s}} \sum\limits_{p_0< p_s} \sum\limits_{\substack{d|P(p_0) \\ ap_0d\ge M}} \mu(d)\sum\limits_{b\in\mathbb{N}} w(abp_0d)\\
 =&\sum\limits_{\substack{(p_1,...,p_s)\in\mathcal{P}_s \\ a=p_1\cdots p_s}}\sum\limits_{n\in\mathbb{N}} \sum\limits_{p_0< p_s}\mathop{\sum\sum}\limits_{\substack{b,d\\ d|P(p_0)\\ bp_0d=n \\ ap_0d\ge M}}\mu(d)w(an).
 \end{split}
 \end{equation*}
 In order to apply \eqref{typeII1}, we must disentangle the variables $a$ and $n$ in the above summation. We have 
 \begin{equation*}
 S_p^{\flat,2}(s,w)=\sum\limits_{\substack{(p_1,...,p_s)\in\mathcal{P}_s \\ a=p_1\cdots p_s}}\sum\limits_{n\in\mathbb{N}} \sum\limits_{p_0< p_s}\mathop{\sum\sum}\limits_{\substack{b,d\\ d|P(p_0)\\ bp_0d=n}} \mu(d)\chi(a,d,p_0,p_s) w(an),
 \end{equation*}
 where 
 \begin{equation*}
    \chi(a,d,p_0,p_s)=1_{\{ap_0d\ge M\}}1_{\{p_0<p_s \}}.
 \end{equation*}
We choose some real number $\rho$ with $|\rho|\leq 1/2$ and $\{M+\rho\}=1/2$, where $\{.\}$ denotes the fractional part. Then the condition $ap_0d\ge M$ is equivalent to $\log(ap_0d)\ge \log(M+\rho)$, and we have 
\begin{equation*}
    |\log(ap_0d)- \log(M+\rho)|\ge \log\frac{x+1}{x+1/2}\ge \frac{1}{3x}.
\end{equation*}
Therefore, Lemma \ref{CoSu} shows that 
\begin{equation*}
    1_{\{ap_0d\ge M\}}=1-\frac{1}{\pi}\int\limits_{-T}^{T}(ap_0d)^{it}\sin(t\log(M+\rho))\frac{dt}{t}+ O\left(\frac{x}{T}\right)
\end{equation*}
for every $T\geq 1$. Similarly,
\begin{equation*}
    1_{\{p_0<p_s\}} =\frac{1}{\pi}\int\limits_{-T}^{T}e^{it/2}e^{itp_0}\sin(tp_s)\frac{dt}{t}+O\left(\frac{1}{T}\right).
\end{equation*}
Thus,
\begin{equation} \label{4.17}
\begin{split}
  S_{p}^{\flat,2}(s,w)=&\frac{1}{\pi}\int\limits_{-T}^{T}\mathop{\sum\sum}\limits_{a,n\in \mathbb{N}}c_a(t)b_n(t)w(an)\frac{dt}{t}\\& - \frac{1}{\pi^2}\int\limits_{-T}^{T}\int\limits_{-T}^{T}\mathop{\sum\sum}\limits_{\substack{a,n\in \mathbb{N}}}c_a(t,\tau)b_n(t,\tau)w(an)\frac{d\tau}{\tau}\frac{dt}{t}\\ &+O\Bigg(\Bigg(\frac{x}{T}+\frac{1}{T}\int\limits_{-T}^{T}|\sin{(\tau\log(M+\rho))}|\frac{d\tau}{\tau}\Bigg)\times\\ &  \Bigg(\sum\limits_{\substack{(p_1,...,p_s)\in\mathcal{P}_s \\ a=p_1\cdots p_s}}\sum\limits_{n\in\mathbb{N}} 
\sum\limits_{p_0< p_s}\mathop{\sum\sum}\limits_{\substack{b,d\\ d|P(p_0)\\ bp_0d=n}}  w(an) \Bigg)\Bigg) 
\end{split}
\end{equation}
with coefficients 
\begin{equation} \label{4.18}
\begin{split}
    c_{a}(t):= &
       \begin{cases}
       \sin(tp_s) & \mbox{ if there is } (p_1,...,p_s)\in\mathcal{P}_s \mbox{ such that } a=p_1\cdots p_s,   \\
       0 & \mbox{ otherwise,}
     \end{cases}\\
    b_n(t):=&\sum\limits_{p_0\in\mathbb{P}(z)}\mathop{\sum\sum}\limits_{\substack{b,d\\ d|P(p_0)\\ bp_0d=n}}e^{it/2}e^{itp_0}\mu(d),\\
    c_{a}(t,\tau):=& c_{a}(t)a^{i\tau}\sin(\tau\log(M+\rho)),\\
    b_{n}(t,\tau):=&\sum\limits_{p_0\in\mathbb{P}(z)}\mathop{\sum\sum}\limits_{\substack{b,d\\ 
d|P(p_0)\\ bp_0d=n}} e^{it/2}e^{itp_0}\mu(d)(p_0d)^{i\tau}.
\end{split}
\end{equation}
Now we set $T:=x^3$. Then the $O$-term in \eqref{4.17} is bounded by 
$$
\ll \frac{1}{x^2} \cdot \sum\limits_{a\in \mathbb{N}} \tau_4(a)w(a)\ll \frac{\mathcal{L}^{10}}{x} \mbox{ if } w=\omega,\tilde{\omega},
$$
where $\tau_k(a)$ is the number of representations of $a\in \mathbb{N}$ as a product of $k$ positive integers. Hence we have
\begin{equation} \label{4.18}
\begin{split}
 & \sum\limits_{p\in \mathcal{P}} \left|S_{p}^{\flat,2}(s)\right|^2\\ \ll & \sum\limits_{p\in \mathcal{S}} \left|\int\limits_{-x^3}^{x^3}\mathop{\sum\sum}\limits_{a,n\in \mathbb{N}}c_a(t)b_n(t)(\omega_p(an)-\tilde{\omega}_p(an))\frac{dt}{t} \right|^2\\& + \sum\limits_{p\in \mathcal{S}} \left|
  \int\limits_{-x^3}^{x^3}\int\limits_{-x^3}^{x^3}\mathop{\sum\sum}\limits_{\substack{a,n\in \mathbb{N}}}c_a(t,\tau)b_n(t,\tau)(\omega_p(an)-\tilde{\omega}_p(an))\frac{d\tau}{\tau}\frac{dt}{t}\right|^2 + 1. 
\end{split}
\end{equation}

We proceed by gathering some intermediate information on the sizes of the coefficients before applying \eqref{typeII1} and the Cauchy-Schwarz inequality. Clearly, 
\begin{equation*}
     |b_n(t)|,|b_n(t,\tau)|\leq \tau(n).
\end{equation*}
For the other coefficients we have
\begin{equation*}
     |c_a(t)|\leq \min\left\{1, \eta_1^{-1}|t|\right\}
\end{equation*}
and 
\begin{equation*}
     |c_a(t,\tau)|\leq \min\left\{1, \eta_1^{-1}|t|, \eta_2^{-1}|\tau|, \eta_1^{-1}\eta_2^{-1}|t\tau|\right\},
\end{equation*}
where $\eta_1:=x^{-1/2}$ and $\eta_2:=(\log(x+1/2))^{-1}$. The small sizes of $c_a(t)$ and $c_a(t,\tau)$ for small $t$'s and $\tau$'s compensate the large values of $1/t$ and $1/\tau$ in the integrands in these cases. 
In view of this, we bound the right-hand side of \eqref{4.18} by 
\begin{equation*}
\begin{split}
\le & \sum\limits_{p\in \mathcal{S}} \left(\int\limits_{-x^3}^{x^3} \max\left\{\eta_1,|t|\right\}^{-1} \left|
 \mathop{\sum\sum}\limits_{a,n\in \mathbb{N}}c_a^{\ast}(t)b_n(t)(\omega_p(an)-\tilde{\omega}_p(an)) \right| dt \right)^2 \\
 + & 
\sum\limits_{p\in \mathcal{S}}
\left( \int\limits_{-x^3}^{x^3}\int\limits_{-x^3}^{x^3} \max\left\{\eta_1,|t|\right\}^{-1} \max\left\{\eta_2,|\tau|\right\}^{-1}
\right.\\ & \left. \times \left|\mathop{\sum\sum}\limits_{\substack{a,n\in \mathbb{N}}}c_a^{\ast}(t,\tau)b_n(t,\tau)(\omega_p(an)-\tilde{\omega}_p(an)) \right| d\tau dt\right)^2 +O(1),  
\end{split}
\end{equation*}
where 
$$
c_a^{\ast}(t):= \begin{cases} \eta_1 c_a(t)/t & \mbox{ if } |t|\le \eta_1 \\ c_a(t) & \mbox{ if } |t|>\eta_1 \end{cases}
$$
and 
$$
c_a^{\ast}(t,\tau):= \begin{cases} \eta_1\eta_2 c_a(t,\tau)/(t\tau) & \mbox{ if } |t|\le \eta_1 \mbox{ and } |\tau|\le \eta_2 \\
\eta_1c_a(t,\tau)/t & \mbox{ if } |t|\le \eta_1 \mbox{ and } |\tau|> \eta_2\\
\eta_2 c_a(t,\tau)/\tau & \mbox{ if } |t|> \eta_1 \mbox{ and } |\tau|\le \eta_2\\
c_a(t,\tau) & \mbox{ if } |t|>\eta_1 \mbox{ and } |\tau|> \eta_2.
\end{cases}
$$
We note that the coefficients $c_a^{\ast}(t)$ and $c_a^{\ast}(t,\tau)$ satisfy
$$
|c_a^{\ast}(t)|\le 1 \mbox{ and } |c_a^{\ast}(t,\tau)|\le 1 \mbox{ for all } t,\tau\in \mathbb{R}.
$$
Now the desired inequality \eqref{4.13} follows by an application of the Cauchy-Schwarz inequality for integrals and \eqref{typeII1}. This completes the proof. $\Box$


\begin{thebibliography}{20}
	
	\bibitem{BM}
	S. Baier, D. Mazumder.
	\newblock \emph{Diophantine approximation with prime restriction 
	in real quadratic number fields}.
	\newblock {Math. Z.} 299, 600--750 (2021).
	
	\bibitem{BT}
	S.~{Baier}, M. Technau.
	\newblock \emph{On the distribution of 
$\alpha p$ modulo one in imaginary quadratic number fields with class number one}.
	\newblock {J. Théor. Nombres Bordx.
 32}, No. 3, 719--760 (2020).
	
	\bibitem{Bak} R.C. Baker. 
\newblock \emph{
A theorem of Bombieri-Vinogradov type with few exceptional moduli}.
\newblock
Acta Arith. 195, No. 3, 313--325 (2020).

\bibitem{BHP} R.C. Baker, G. Harman, J. Pintz.
\newblock \emph{The difference between consecutive primes. II}. 
\newblock
Proc. Lond. Math. Soc., III. Ser. 83, No. 3, 532--562 (2001).

    \bibitem{BaS} W.D. Banks, I.E. Shparlinski.
    \newblock \emph{Bounds on short character sums and L-functions with characters to a powerful modulus}. 
    \newblock J. Anal. Math. 139, No. 1, 239--263 (2019).
	
	\bibitem{BLC} M.B. Barban, Yu.V. Linnik, N.G. Chudakov.
    \newblock \emph{On prime numbers in an arithmetic progression with a prime-power difference}.
    \newblock Acta Arith. 9, 375--390 (1964).
	
	\bibitem{Bru} J. Br\"udern.
    \newblock \emph{Einf\"uhrung in die analytische Zahlentheorie.} [\emph{Introduction to analytic number theory} (German)], Berlin: Springer-Verlag. x, 238 p. (1995).
	
    \bibitem{Ell}	P.D.T.A Elliott.
    \newblock \emph{Primes in progressions to moduli with a large power factor}. 
    \newblock Ramanujan J. 13, No. 1--3, 241-251 (2007).
    
	\bibitem{Gal} P.X. Gallagher. 
    \newblock \emph{Primes in progressions to prime-power modulus}.
    \newblock Invent. Math. 16, 191--201 (1972).
    
   \bibitem{Guo} R. Guo. 
   \newblock \emph{Primes in Arithmetic Progressions to Moduli with a Large Power Factor}.
   \newblock Advances in Pure Mathematics, 3, 25--32 (2013).
	
	\bibitem{Har}
	G.~{Harman}.
	\newblock \emph{On the distribution of $\alpha p$ modulo one. II}.
	\newblock {Proc.\ London Math.\ Soc.  72}, No. 3, 241--260 (1996).
	
	\bibitem{HarCarmI}
	G. Harman.
    \newblock \emph{On the number of Carmichael numbers up to $x$}.
    \newblock Bull. Lond. Math. Soc. 37, No. 5, 641--650 (2005).

    \bibitem{HarPri}
	G.~{Harman}.
	\newblock \emph{Prime-detecting sieves}.
	\newblock Princeton, NJ: Princeton University Press (2007).
	
	\bibitem{HarCarmII}
	G. Harman. 
	\newblock \emph{Watt’s mean value theorem and Carmichael numbers}.
    \newblock Int. J. Number Theory 4, No. 2, 241--248 (2008).
    
    \bibitem{HarQi}
    G. Harman.
    \newblock \emph{Diophantine approximation with Gaussian primes}.
    \newblock Q. J. Math. 70, No. 4, 1505--1519 (2019)
    
	\bibitem{Hux} M.N. Huxley. 
	\newblock \emph{On the difference between consecutive primes}. 
	\newblock Invent. Math., 15, 164--170 (1972).
	
	\bibitem{Hux2} M.N. Huxley, 
\newblock \emph{Large values of Dirichlet polynomials. III}.
    \newblock Acta Arith. 26, 435-444 (1975).
	
	\bibitem{Iwa} H. Iwaniec.
	\newblock \emph{On zeros of Dirichlet’s L-series}. 
    \newblock Invent. Math. 23, 97--104 (1974).

	\bibitem{Mat} K. Matom\"aki. 
	\newblock \emph{The distribution of $\alpha p$ modulo one.} 
	\newblock {Math. Proc. Camb. Philos. Soc. 147}, No. 2, 267--283 (2009).
	
	\bibitem{Pos} A. G. Postnikov.
	\newblock \emph{\"Uber die Summe der Charaktere nach einem Primzahlpotenzmodul}. 
	\newblock Izv. Akad. Nauk SSSR, Ser. Mat. 19, No. 1, 11--16 (1955).
	
    \newblock Graduate Studies in Mathematics 163. Providence, RI: American Mathematical Society (AMS) (2015).
        
    \bibitem{Wat1} N. Watt,
    \newblock \emph{Kloosterman sums and a mean value for Dirichlet     polynomials}.
    \newblock J. Number Theory 53, No. 1, 179--210 (1995).
    
    \bibitem{Wat2} N. Watt,
    \newblock \emph{Bounds for a mean value of character sums}.
    \newblock Int. J. Number Theory 4, No. 2, 249--293 (2008).

    \bibitem{Xyl} T. Xylouris.
    \newblock \emph{\"Uber die Nullstellen der Dirichletschen L-Funktionen und die kleinste Primzahl in einer arithmetischen Progression} [\emph{The zeros of Dirichlet L-functions and the least prime in an arithmetic progression}]. 
    \newblock {Dissertation for the degree of Doctor of Mathematics and Natural Sciences (in German)}, Bonn: Universit\"at Bonn, Mathematisches Institut. 
\end{thebibliography}
\end{document}